\documentclass[a4paper,11pt,twoside]{article}

\usepackage{amssymb,amsmath,amsfonts, amsthm,amscd}

\usepackage{verbatim}

\input xy
\xyoption{all}

\newtheorem{theorem}{Theorem}[section]
\newtheorem{lemma}[theorem]{Lemma}
\newtheorem{proposition}[theorem]{Proposition}
\newtheorem{corollary}[theorem]{Corollary}

\theoremstyle{definition}

\newenvironment{rem}[1][Remark.]{\begin{trivlist}
\item[\hskip \labelsep {\bfseries #1}]}{\end{trivlist}}

\newenvironment{notation}[1][Notation.]{\begin{trivlist}
\item[\hskip \labelsep {\bfseries #1}]}{\end{trivlist}}

\usepackage{hyperref}
\hypersetup{
colorlinks=true,
linkcolor=black,
citecolor=black
}

\newcommand{\Lie}{\mathfrak}

\title{Local coordinates  for $\operatorname{SL}(n,\mathbf C)$ character varieties of finite volume hyperbolic 3-manifolds.}
\author{Pere Menal-Ferrer \and Joan Porti 
\thanks{Both authors partially supported by  
FEDER/Micinn through grant MTM2009-0759 and by 
 AGAUR through grant SGR2009-1207. 
The second author received the prize “ICREA Acad\`emia” 
for excellence in research, funded by the Generalitat de Catalunya.}
}

\begin{document}
\maketitle

\begin{abstract}
Given a finite volume hyperbolic 3-manifold, we compose a lift of the holonomy in $\operatorname{SL}(2,\mathbf C)$
with the $n$-dimensional irreducible representation of $\operatorname{SL}(2,\mathbf C)$ in $\operatorname{SL}(n,\mathbf C)$. In this paper we
give local coordinates of the $\operatorname{SL}(n,\mathbf C)$-character variety around this representation.
As a corollary, this representation is isolated among all representations that are unipotent at the cusps.
\end{abstract}


\section{Introduction}

Let $M^3$ be an orientable hyperbolic 3-manifold of finite volume with $l> 0$ ends, that are  cusps.
This manifold is homeomorphic to the interior of a compact manifold $\overline M^3$ with
boundary a union of $l$ tori.
Let 
$\widetilde{\operatorname{Hol}}\colon \pi_1(M^3) \rightarrow \operatorname{SL}(2,\mathbf{C})$ 
be a lift of the holonomy of $M^3$ and compose it with the irreducible $n$-dimensional representation
$$
{\varsigma_n}:\operatorname{SL}(2,\mathbf C)\to \operatorname{SL}(n,\mathbf C).
$$
 The composition is denoted by 
\[
	\rho_n= {\varsigma_n}\circ \widetilde{\operatorname{Hol}} \colon \pi_1(M^3) \rightarrow \operatorname{SL}(n, \mathbf{C}).
\]
The variety of characters
\(
      X(M^3,\operatorname{SL}(n,\mathbf C))
\) 
is the algebraic quotient of the variety of representations
$\hom(\pi_1(M^3),\operatorname{SL}(n,\mathbf C))$ by the action of $\operatorname{SL}(n,\mathbf C)$ by conjugation \cite{LubotzkiMagid}.
The character of $\rho_n$ will be denoted by $\chi_n$. 
In \cite{MenalPorti} it is proved that  $\chi_n$ is a smooth point of $X(M^3,\operatorname{SL}(n,\mathbf C))$, with local dimension $(n-1)l$, where $l$ is the number of ends of $M^3$.
The goal of this paper is to find coordinates for a neighborhood of $\chi_n$.

For $i=1,\ldots,n-1$, let
\[
\sigma_i\colon  \operatorname{SL}(n,\mathbf C)\to \mathbf C
\]
denote the $i$-th elementary symmetric polynomial on the eigenvalues, so that the characteristic polynomial of $A\in  \operatorname{SL}(n,\mathbf C)$ is
\[
 P_A(\lambda)=\lambda^n-\sigma_1(A)\lambda ^{n-1}+\cdots + (-1)^{n-1}\sigma_{n-1}(A) \lambda+(-1)^n.
\]
So $\sigma_1(A)$ is the trace of $A$, $\sigma_2(A)$ is obtained from $2\times 2$ principal minors of $A$, and so on. 
The $\sigma_i\colon  \operatorname{SL}(n,\mathbf C)\to \mathbf C$ are polynomial functions, invariant by conjugation. Thus, for any $\gamma\in\pi_1(M^3)$, the map 
\[
 \begin{array}{rcl}
  \hom(\pi_1(M^3), SL(n,\mathbf C))  &   \to   & \mathbf C \\
       \rho     & \mapsto & \sigma_i(\rho(\gamma))
 \end{array}
\]
induces a polynomial map on the character variety
\[
\sigma_i^{\gamma} : X(M^3,\operatorname{SL}(n,\mathbf C))\to\mathbf C.
\]
The main result of this paper is  the following theorem.

\begin{theorem}
 \label{thm:main}
Let $M^3$ be a finite volume, orientable, hyperbolic 3-manifold with  $l> 0$ ends.
Let $\gamma_1,\ldots,\gamma_l \in\pi_1(M^3)$ be nontrivial peripheral elements, one for each end of $M^3$ (or each boundary component of  $\overline M^3$).
Then 
\[
(\sigma_1^{ \gamma_1}, \ldots,\sigma_{n-1}^{\gamma_1},\ldots,\sigma_1^{ \gamma_l}, \ldots,\sigma_{n-1}^{\gamma_l}) \colon X(M^3,\operatorname{SL}(n,\mathbf C))\to\mathbf C^{l(n-1)}
\]
is a local biholomorphism at $\chi_n$.
\end{theorem}

\begin{corollary}
The character $\chi_n$ is isolated among all characters of representations in $ \operatorname{SL}(n,\mathbf C)$ that are unipotent at the peripheral subgroups.
\end{corollary}

When $n=2$, we obtain the following result of Kapovich \cite{Kapovich} (see also Bromberg \cite{Bromberg}).

\begin{corollary} The deformation space $X(M^3,\operatorname{SL}(2,\mathbf C))$ is locally parameterized by the trace of $\gamma_1,\ldots,\gamma_l$ around a 
lift of the holonomy representation.
\end{corollary}

The proof relies on a vanishing theorem of \cite{RagAJM,MatMur}, that asserts that infinitesimal $L^2$ deformations are trivial. In this way we
determine explicit differential forms on the cusp that describe the infinitesimal deformations and prove Theorem~\ref{thm:main}.

The paper is organized as follows. In Section~\ref{sec:rel} we describe the basic facts of
 the n-dimensional irreducible representation 
${\varsigma_n}:\operatorname{SL}(2,\mathbf C)\to \operatorname{SL}(n,\mathbf C)$,
that will be required later. Section~\ref{sec:cohomology}
is devoted to compute the explicit differential forms that give the infinitesimal deformations. Finally, in Section~\ref{section:derivate}
we compute the derivative of the $\sigma_i^{\gamma_j}$ with respect to these infinitesimal deformations.

\section{The $n$-dimensional representation}
\label{sec:rel}

The irreducible $n$-dimensional complex representation 
$$
{\varsigma_n}\colon \operatorname{SL}(2,\mathbf C) \to \operatorname{SL}(n,\mathbf C)
$$
is the $(n-1)$-symmetric power $\operatorname{Sym}^{n-1}(\mathbf C^2)\cong\mathbf C^{n}$.
The induced representation of Lie algebras is also denoted by
 ${\varsigma_n}: \Lie{sl}(2,\mathbf C) \to \Lie{sl}(n,\mathbf C)$.
We shall work with the  basis for $\Lie{sl}(2,\mathbf C)$  given by
\[
\Lie{e} = \begin{pmatrix}
    1 & 0 \\
    0 & -1
   \end{pmatrix}, \quad
\Lie{f} = \begin{pmatrix}
    0 & 1 \\
    0 & 0
   \end{pmatrix}, \quad
\Lie{g} = \begin{pmatrix}
    0 & 0 \\
    1 & 0
   \end{pmatrix}.
\]
A straightforward computation shows that:
\begin{equation} \label{eqn:h+}
{\mathfrak {h}}_+ := {\varsigma_n}( \Lie{f} ) =  \begin{pmatrix}
                                           0    &    1   &  0  &      &    &  \\
					   0    &    0   &  2  &   &    &  \\
					   0    &    0   &  0  &         &    &   \\
					       &  &     & \ddots  &    &   \\
					       &       &    &         & 0  &  n-1 \\
					       &       &    &   & 0  &  0 
                                          \end{pmatrix},
\end{equation}
namely the $(i,j)$-entry of $\mathfrak {h}_+ $ is $i$ when $j=i+1$ and $0$ otherwise.
Similarly
\begin{equation} \label{eqn:h-}
{\mathfrak {h}}_- := {\varsigma_n}( \Lie{g} ) = \begin{pmatrix}
                                           0    &    0   &   0 &         &    &    \\
					   n-1  &    0   &   0 &         &    &    \\
					   0    &  n-2   &   0 &         &    &    \\
					        &        &     & \ddots  &    &    \\
					        &        &     &         & 0  &  0 \\
					        &        &     &         & 1  &  0 
                                          \end{pmatrix}.
\end{equation}
This allows to describe $\varsigma_n $ for $\pm \exp(\beta \Lie f)$ and $\pm \exp(\beta \Lie g)$, $\beta\in\mathbf C$:
\begin{eqnarray} \label{eqn:sym(m)}
{\varsigma_n }(
 \pm\left(\begin{smallmatrix}
    1 & \beta \\
    0 & 1
   \end{smallmatrix}\right)
)
& = &(\pm 1)^{n-1} \varsigma_n( e^{\beta  \Lie{f}} ) =  (\pm 1)^{n-1} e^{ \beta \Lie{h}_+ }, \\
{\varsigma_n }(
 \pm\left(\begin{smallmatrix}
    1 & 0 \\
    \beta & 1
   \end{smallmatrix}\right)
)
& = &(\pm 1)^{n-1} \varsigma_n( e^{\beta  \Lie{g}} ) =  (\pm 1)^{n-1} e^{ \beta \Lie{h}_-}. 
\end{eqnarray}
 Notice that both matrices are triangular   with $1$ in the diagonal, in particular unipotent.

\begin{notation}
	The group $\operatorname{SL}(2,\mathbf C)$ acts on the Lie algebra 
	$\mathfrak{sl}(n,\mathbf C)$ by composing the adjoint representation with ${\varsigma_n}$. For 
	$A\in \operatorname{SL}(2,\mathbf C)$ and $\mathfrak{a}\in \mathfrak{sl}(n,\mathbf C)$, 
	this action will be simply denoted by $A\, \mathfrak{a}$. Namely,
	\[
	A\, \mathfrak{a}= \operatorname{Ad}_{{\varsigma_n}(A)}(\mathfrak{a})= {\varsigma_n}(A)\,\mathfrak{a}\,{\varsigma_n}(A^{-1} ).
	\]
\end{notation}

\begin{lemma} \label{lemma:invariants}
	For $\beta\neq 0$, the subspace of matrices in $\mathfrak{sl}(n,\mathbf C)$ that are invariant by 
	$\pm\left(\begin{smallmatrix}
		1 & \beta \\
		0 & 1
	   \end{smallmatrix}
	   \right)$ is 
	   \[
	   \langle  {\mathfrak {h}}_+,  {\mathfrak {h}}_+^2,\ldots,  {\mathfrak {h}}_+^{n-1}\rangle.
	   \]
\end{lemma}

\begin{proof}
	Since $\pm\left(\begin{smallmatrix}
		1 & \beta \\
		0 & 1
	   \end{smallmatrix}
	   \right)
	=\exp(\beta\Lie{g})$,
	the lemma is equivalent to saying that $ \langle  {\mathfrak {h}}_+,  {\mathfrak {h}}_+^2,\ldots,  {\mathfrak {h}}_+^{n-1}\rangle $ 
	is the subspace of invariants by the one-parameter group
	generated by $\Lie{f}$, and this latter space is exactly 
	\[
	\operatorname{Ker} \Lie{f} = \left\{ \mathfrak{a} \in \mathfrak{sl}(n,\mathbf C) \mid [\varsigma_n(\Lie{f}), \mathfrak{a}] = 0 \right\}.
	\]
	Since ${\mathfrak {h}}_+ = {\varsigma_n}( \Lie{f} ) $,
	$$
	     \langle  {\mathfrak {h}}_+,  {\mathfrak {h}}_+^2,\ldots,  {\mathfrak {h}}_+^{n-1}\rangle \subseteq  \operatorname{Ker} \Lie{f}.
	$$
	To prove the equality, we show that $\operatorname{Ker} \Lie{f}$ has dimension $n-1$.
	To see this, we decompose $\mathfrak{sl}(n,\mathbf C)$, as  $\operatorname{SL}(2,\mathbf C)$-module,
	into irreducible factors using Clebsh-Gordan:
	\[ 
	\mathfrak{sl}(n,\mathbf C) = \operatorname{Sym}^{2n-1}(\mathbf C^2)\oplus \ldots \oplus \operatorname{Sym}^{5}(\mathbf C^2)
	\oplus \operatorname{Sym}^{3}(\mathbf C^2).
	\] 
	As an endomorphism of $\operatorname{Sym}^{k}(\mathbf C^2)$, the rank of $\Lie{f}$ is $k$ 
	(use for instance Equation \eqref{eqn:h-}), and hence its kernel has dimension $1$. 
	The result then follows immediately.
\end{proof}

We shall also require the following computations. Since $[\Lie{e},\Lie{f}]= 2\Lie{f}$ and $[\Lie{e},\Lie{g}]= -2\Lie{g}$,
\begin{equation}
	\begin{pmatrix} \lambda & 0 \\ 0 & 1/\lambda\end{pmatrix}\,  {\mathfrak {h}}_{\pm}= \lambda ^{\pm 2} \,  {\mathfrak {h}}_{\pm}.
\end{equation}
Hence, for $i=1,\ldots,n-1$,
\begin{equation}
	\label{eqn:lambdah+}
	\begin{pmatrix} \lambda & 0 \\ 0 & 1/\lambda\end{pmatrix}\,  {\mathfrak {h}}_+^i= \lambda ^{2i} \,  {\mathfrak {h}}_+^i 
\qquad \textrm{ and } \qquad 
	\begin{pmatrix} \lambda & 0 \\ 0 & 1/\lambda\end{pmatrix}\,  {\mathfrak {h}}_-^i= \lambda ^{-2i} \,  {\mathfrak {h}}_-^i.
\end{equation}
Finally, we recall the bilinear product
\begin{equation}
\label{eqn:bilinear}
 \begin{array}{ccl}
  \mathfrak{sl}(n,\mathbf C)\times   \mathfrak{sl}(n,\mathbf C) & \to & \mathbf C \\
         ({\mathfrak v}_1,{\mathfrak v}_2)& \mapsto &\operatorname{trace}({\mathfrak v}_1\, {\mathfrak v}_2),
 \end{array}
\end{equation}
which is a multiple of the Killing form.
This pairing  is nondegenerate, symmetric,  bilinear and 
$\operatorname{Ad}\circ{\varsigma_n}$-invariant
(hence $\operatorname{Ad}\circ\rho_n$-invariant). In addition
$\operatorname{trace}(  {\mathfrak {h}}_-^i\,  {\mathfrak {h}}_+^j)=0$ iff $i\neq j$. For $i=1,\ldots,n-1$, we denote $c_i=\operatorname{trace}(  {\mathfrak {h}}_-^i\,  {\mathfrak {h}}_+^i)\neq 0$.
Such a pairing is not unique, as $\mathfrak{sl}(n,\mathbf C)$ is not an irreducible $\operatorname{SL}(2,\mathbf{C})$-module.

\section{Infinitesimal deformations}
\label{sec:cohomology}

To describe explicit infinitesimal deformations, we shall work in cohomology with twisted coefficients. 
The representation $\rho_n$ is semisimple, hence by \cite{Weil,LubotzkiMagid} the Zariski tangent space of $X(M^3,\operatorname{SL}(n,\mathbf C))$ at 
$\chi_n$ is isomorphic to
$H^1(\pi_1(M^3), \mathfrak{sl}(n,\mathbf C)_{\operatorname{Ad}\circ\rho_n})$, where $ \mathfrak{sl}(n,\mathbf C)_{\operatorname{Ad}\circ\rho_n}$ denotes the Lie algebra with the action obtained by 
composing  $\rho_n$ and the adjoint representation 
(this is described in more detail in  Section~\ref{section:derivate}, cf.\ \eqref{eqn:weil}).

Since $M^3$ is aspherical, we shall work with the cohomology of $M^3$ with coefficients in the flat bundle 
\[
E_{\operatorname{Ad}\circ\rho_n}= \widetilde{M^3} \times_{\pi_1M^3} \mathfrak{sl}(n,\mathbf C)_{\operatorname{Ad}\circ\rho_n}.
\]
Let $\Omega^p(M^3,E_{\operatorname{Ad}\circ\rho_n})$ denote the space of smooth p-forms on $M^3$ valued on $E_{\operatorname{Ad}\circ\rho_n}$, 
ie. the space of smooth sections of the bundle $\bigwedge^p T^* M^3\otimes  E_{\operatorname{Ad}\circ\rho_n}$. The de Rham cohomology of 
$\Omega^*(M^3,E_{\operatorname{Ad}\circ\rho_n})$
is denoted by 
$$
H^*(M^3; E_{\operatorname{Ad}\circ\rho_n}),
$$
and it is naturally isomorphic to  the group cohomology
$$
H^*(\pi_1(M^3), \mathfrak{sl}(n,\mathbf C)_{\operatorname{Ad}\circ\rho_n}).
$$
The explicit constructions of these identifications will be given in Section~\ref{section:derivate}.

Next we need to recall the  inner product on $\Omega^p(M^3,E_{\operatorname{Ad}\circ\rho_n})$. In order to do that, 
start with the homogeneous structure of hyperbolic space,
$$
\mathbf H^3\cong \widetilde{M^3}=\operatorname{SL}(2,\mathbf C)/\operatorname{SU}(2),
$$
ie.\ $\operatorname{SU}(2)$ is the stabilizer of a base point $p\in\mathbf H^3$. 
Fix a $\operatorname{SU}(2)$-invariant hermitian
product on $\mathfrak{sl}(n,\mathbf C)$, that we choose to be  the product  $\langle\;,\rangle_p$
at the fiber $ E_{\operatorname{Ad}\circ\rho_n,p}$ of the base point $p$. 
Use the rule
$$
\langle  v_1, v_2\rangle_{\gamma p}= \langle \gamma^{-1}\,v_1,\gamma^{-1}\, v_2\rangle_{p},\qquad \forall \gamma\in \operatorname{SL}(2,\mathbf C), \ v_1,v_2\in E_{\operatorname{Ad}\circ\rho_n,\gamma p},
$$
to define it at the fiber of any point $\gamma p\in\mathbf H^3$.
By using an orthonormal basis, it induces an inner
product on the fibers of $\bigwedge^* T^* M^3\otimes  E_{\operatorname{Ad}\circ\rho_n}$, and on $\Omega^*(M^3; E_{\operatorname{Ad}\circ\rho_n})$
by integration:
 if $\alpha,\beta \in \Omega^*(M^3; E_{\operatorname{Ad}\circ\rho_n})$ then
$$(\alpha,\beta) = \int_M \langle\alpha(x),\beta(x)\rangle_x d\operatorname{vol}.$$
A form $\alpha\in \Omega^*(M^3; E_{\operatorname{Ad}\circ\rho_n})$ is $L^2$ if $\vert\alpha\vert^2=(\alpha,\alpha)<\infty$.

Recall that $M^3$ has $l$ cusps and that it is homeomorphic to the interior of a compact manifold $\overline M^3$ with
boundary a union of $l$ tori. We shall use the following result from \cite{MenalPorti}, based on rigidity results of Raghunathan
\cite{RagAJM} and Mathsushima-Murakami \cite{MatMur}.

\begin{theorem}[\cite{MenalPorti}]
	\label{thm:l2rigidity}
	If $n\geq 2$, then 
	\[
	\dim_\mathbf{C} H^1(M^3; E_{\operatorname{Ad}\circ \rho_n} ) = l(n-1).
	\]
	In addition, all nontrivial elements in $H^1(M^3; E_{\operatorname{Ad}\circ \rho_n} )$
	are nontrivial in $H^1(\partial\overline M^3 ; E_{\operatorname{Ad}\circ \rho_n} )$ and have no $L^2$-representative.
\end{theorem}

\begin{rem} To simplify, from now on we will assume that $l=1$, ie. $M^3$ \emph{has a single cusp}. The proof below applies to any $l\in\mathbf N$, 
because most of the argument is localized at the cusp.
\end{rem}

We need to describe the metric on a cusp $U\subset M^3$, namely $U$ is the quotient of a horoball in $\mathbf H^3$ by a rank
two parabolic group of isometries. Notice that $M\setminus\operatorname{int}(U)$ is compact, thus a form $\Omega^*(M^3,E_{\operatorname{Ad}\circ \rho_n})$
is $L^2$ (has finite $L^2$ norm) if and only if its restriction to $\Omega^*(U,E_{\operatorname{Ad}\circ \rho_n})$ is $L^2$.

The cusp $U$ is diffeomorphic  to $T^2\times [0,\infty)$, and it is isometric to the warped product
$$
d t^2 + e^{-2 t} d s^2_{T^2} ,
$$ 
where $d s^2_{T^2}$ denotes a flat metric on the $2$-torus.
Consider $\vartheta$ any 1-form on the $2$-torus $T^2$, and view it as a form on $U$ by pullback from the projection
to the first factor
$U= T^2\times [0,\infty)\to T^2$.

Assume that the holonomy of the cusp lies in the group
$$
\left\lbrace
\pm
\begin{pmatrix}
 1 & a \\
 0 & 1
\end{pmatrix}
\mid 
a\in\mathbf C
\right\rbrace
.
$$
Recall that $ {\mathfrak {h}}_+\in\mathfrak{sl}(n,\mathbf C)$ is defined in \eqref{eqn:h+}, and that $ {\mathfrak {h}}_+^j$ is invariant by the 
holonomy of the cusp, for $j=1,\ldots,n-1$, by Lemma~\ref{lemma:invariants}.
In particular, the form $\vartheta\otimes {\mathfrak {h}}_+^ j$ is well defined, and it is closed iff $\vartheta$ is closed.

\begin{lemma}
\label{Lemma:forml2}
 The 1-form $\vartheta\otimes  {\mathfrak {h}}_+^j$ is $L^2$, for $j=1,\ldots,n-1$.
 \end{lemma}

\begin{proof}
Given $p\in T^2$  and $t\in [0,\infty)$, we first compute the norm of 
$\vartheta\otimes  {\mathfrak {h}}_+^j$ at $(p,t)\in T^2\times [0,\infty)=U$, and then we shall show that 
\begin{equation*}
\int_U \vert \vartheta\otimes    {\mathfrak {h}}_+^j\vert^2_{(p,t)} d\operatorname{vol}_U<\infty.
\end{equation*}
By compactness, there exists a constant $C>0$ such that $\vert \vartheta\vert_{(p,0)}\leq C$ for every point $p\in T^2$
when $t=0$. Since the metric is the warped product $ d t^2 + e^{-2 t} d s^2_{T^2}$:
$$
\vert \vartheta\vert_{(p,t)}\leq e^t C.
$$
On the other hand, if we work in the half space model for $\mathbf H^3$ and we assume that the horoball is centered at $\infty$,
the image of $\pi_1(U)$ is contained in $\pm\left(\begin{smallmatrix}
                                             1 & * \\ 0 & 1
                                            \end{smallmatrix}\right)
$.
Then
the isometry that brings a base point  to a lift of $(p,t)$ in $\mathbf H^3$ is 
$$
\begin{pmatrix}
 e^{t/2} & z e^{-t/2} \\
 0 & e^{-t/2}
\end{pmatrix},
$$
for some $z\in\mathbf C$.
By \eqref{eqn:lambdah+} and Lemma~\ref{lemma:invariants}, we have
\[
\begin{pmatrix}
 e^{t/2} & z e^{-t/2} \\
 0 & e^{-t/2}
\end{pmatrix}^{-1}\,  {\mathfrak {h}}_+^j=
\begin{pmatrix}
 e^{-t/2} & 0\\
 0 & e^{t/2}
\end{pmatrix} \,  {\mathfrak {h}}_+^j= e^{-j\, t}  {\mathfrak {h}}_+^j.
\]
By definition of the metric on the bundle $E_{\operatorname{Ad}\circ\rho_n}$:
$$
\vert  {\mathfrak {h}}_+^j\vert_{ (p,t)}=   e^{-j\, t} \vert    {\mathfrak {h}}_+^j   \vert_{ (p,0)}.
$$
Thus 
$$
\vert \vartheta\otimes        {\mathfrak {h}}_+^j     \vert _{(p,t)}\leq C' e^{(1-j)t },
$$
for some constant $C'>0$. In addition, using
$
d\operatorname{vol}_U=  e^{-2t} d\operatorname{vol}_{T^2}\wedge dt,
$
we compute:
$$
\int_U \vert \vartheta\otimes  {\mathfrak {h}}_+^j\vert^2_{(p,t)} d\operatorname{vol}_U\leq 
C''\int_0^{+\infty} e^{2(1-j)t-2 t} dt= C'' \int_0^{+\infty} e^{-2j\, t } dt<+\infty.
$$
\end{proof}

We next look for a basis for $H^1(U;E_{\operatorname{Ad}\circ\rho_n})$ (Lemma~\ref{Lemma:basis} below). 
Choose coordinates 
$(x,y)\in\mathbf R^2$ and view the torus as the quotient $\mathbf R^2/\mathbf Z^2$.
Let $\gamma_1$ and $\gamma_2$ be two generators of $\pi_1(T^2)$, and assume that they
act on the universal covering as:
$$
\gamma_1(x,y)=(x+1,y),\qquad \gamma_2(x,y)=(x,y+1), \qquad \forall x,y\in\mathbf R^2.
$$
Assume that their holonomy is defined by
$$
\gamma_1\to \pm \begin{pmatrix}
                              1 & 1 \\
			      0 & 1
                             \end{pmatrix} \quad\textrm{ and }\quad
\gamma_2\to \pm \begin{pmatrix}
                              1 & \tau \\
			      0 & 1
                             \end{pmatrix} ,
$$
for some $\tau\in\mathbf C\setminus\mathbf R$.

\begin{lemma}
\label{Lemma:basis}
For $i=1,\ldots,n-1$, the form  $$
\omega_i= d(x+\tau y)\otimes \begin{pmatrix}
                              1 & x+\tau y \\
			      0 & 1
                             \end{pmatrix}  {\mathfrak {h}}_-^i \in  \Omega^1(U;E_{\operatorname{Ad}\circ\rho_n})
$$
is closed. 
Moreover, for any $a,b\in\mathbf C$ such that $b\neq a\, \tau$,
$$
\{\omega_1,\ldots,\omega_{n-1},(a\,d x+b\, d y)\otimes  {\mathfrak {h}}_+,\ldots, (a\,d x+b\, d y)\otimes  {\mathfrak {h}}_+^{n-1}\}
$$ is a basis for $H^1(U;E_{\operatorname{Ad}\circ\rho_n})$.
\end{lemma}

\begin{proof} Notice that $\omega_i$ is defined on the universal covering $\widetilde U$ and, by construction, it is equivariant, hence defined on $U$.
 The form $\omega_i$ is closed because $\begin{pmatrix}
                              1 & x+\tau y \\
			      0 & 1
                             \end{pmatrix}  {\mathfrak {h}}_-^i$ 
has coordinates that are polynomial on the function $x+\tau y $, with respect to any $\mathbf C$-basis for $\mathfrak{sl}(n,\mathbf C)$.

Next we want to describe the basis for $H^1(U;E_{\operatorname{Ad}\circ\rho_n})$. 
Knowing that $\dim H^1(U;E_{\operatorname{Ad}\circ\rho_n})=2(n-1)$ \cite{MenalPorti}, we will show that the $2(n-1)$ differential
forms are linearly independent cohomology classes by using a bilinear pairing. 

This pairing is induced from the exterior product 
$\wedge: \Omega^i(U;E_{\operatorname{Ad}\circ\rho_n})\times  \Omega^j(U;E_{\operatorname{Ad}\circ\rho_n})\to \Omega^{i+j}(U;\mathbf C)$
$$
(\vartheta_1\otimes {\mathfrak v}_1)\wedge (\vartheta_2\otimes {\mathfrak v}_2)=  \operatorname{trace}({\mathfrak v}_1\, {\mathfrak v}_2)\, \vartheta_1\wedge\vartheta_2,
$$
where
$\vartheta_1\in \Omega^i(U)$, $\vartheta_2\in \Omega^j(U)$ are forms without coefficients (or coefficients in the trivial bundle), ${\mathfrak v}_1,{\mathfrak v}_2\in \mathfrak{sl}(n,\mathbf C)$. 
Recall that the pairing
$({\mathfrak v}_1,{\mathfrak v}_2)\mapsto \operatorname{trace}({\mathfrak v}_1\, {\mathfrak v}_2)$ was described in \eqref{eqn:bilinear} and that 
$\operatorname{trace}(  {\mathfrak {h}}_-^i\,  {\mathfrak {h}}_+^j)=\delta^j_i c_i$, where $\delta^i_i=1$, $\delta_i^j=0$ for $i\neq j$, and $c_i\neq 0$.

This exterior product induces a cup product in cohomology. Since the pairing and $ {\mathfrak {h}}_+^i$ are both invariant by the action of
$\left(\begin{smallmatrix}
        1 & * \\
	0 & 1
       \end{smallmatrix}
\right)
$,  we have:
\begin{multline*}
( (a\,dx+ b\,dy)\otimes  {\mathfrak {h}}_+^i )\wedge \omega_j = \\  \operatorname{trace}(  {\mathfrak {h}}_+^i \, \begin{pmatrix}
                              1 & x+\tau y \\
			      0 & 1
                             \end{pmatrix} 
 {\mathfrak {h}}_-^j) (a\,dx+ b\,dy)\wedge (dx+\tau dy)       = 
\\ 
\operatorname{trace}(    {\mathfrak {h}}_+^i\, 
  {\mathfrak {h}}_-^j ) (a\,dx+ b\,dy)\wedge (dx+\tau dy)
=  c_i \delta_i^j(a \tau -b)  d x\wedge dy.
\end{multline*}
In addition:
\begin{equation*}
  \omega_i\wedge \omega_j = \operatorname{trace}(   {\mathfrak {h}}_-^i\,   {\mathfrak {h}}_-^j ) (dx+\tau dy) \wedge (dx+\tau dy) =0, 
\end{equation*}
and
\begin{equation*}
   (a\,dx+ b\,dy)\otimes  {\mathfrak {h}}_+^i \wedge  (a\,dx+ b\,dy)\otimes  {\mathfrak {h}}_+^j=0.
\end{equation*}
Since $dx\wedge d y$ is the volume form of the torus, the lemma follows.
\end{proof}

\begin{proposition}
\label{prop:generatorimage}
 The image of the map $H^1(M^3;E_{\operatorname{Ad}\circ\rho_n})\to H^1(U;E_{\operatorname{Ad}\circ\rho_n})$ is the $(n-1)$-dimensional linear span of
$$
\{\omega_1+\sum_j a_{1,j }\, dx\otimes  {\mathfrak {h}}_+^j,\ \ldots\ , \omega_{n-1}+\sum_j a_{n-1,j }\, dx\otimes  {\mathfrak {h}}_+^j\},
$$ 
for some $a_{i,j}\in\mathbf C$.
\end{proposition}

\begin{proof}
By contradiction: if the lemma was not true, then there would be a nontrivial 
element
$
\sum_j  a_{j }\, dx\otimes  {\mathfrak {h}}_+^j
$
 in the image, because it is $n-1$ dimensional. 
 But this form is $L^2$ (Lemma~\ref{Lemma:forml2}) contradicting 
Theorem~\ref{thm:l2rigidity}.
\end{proof}

\section{Derivating the elementary symmetric polynomials}
\label{section:derivate}

We want to compute the derivatives of the elementary symmetric polynomials of a peripheral element $\gamma$ with respect to the infinitesimal deformations 
of Proposition~\ref{prop:generatorimage}.

We first describe the map between closed 1-forms in $\Omega^1(M^3;E_{\operatorname{Ad}\circ\rho_n})$ and group cocycles in
\begin{multline*}
Z^1(M^3;\mathfrak{sl}(n,\mathbf C)_{\operatorname{Ad}\circ\rho_n})=
\{d:\pi_1(M^3)\to \mathfrak{sl}(n,\mathbf C) \mid \\ d(\gamma_1\gamma_2)=d(\gamma_1)+ \operatorname{Ad}_{\rho_n(\gamma)}( d(\gamma_2)),\ \forall\gamma_1,\gamma_2\in\pi_1(M^3)\} 
\end{multline*}
 that induces the isomorphism between de Rham 
and group cohomology.
For this purpose we fix a point $p\in M^3$, that will be the base point for $\pi_1(M^3,p)$. 
Let $\vartheta\in \Omega^1(M^3;E_{\operatorname{Ad}\circ\rho_n})$ be a closed 1-form. In particular it represents an element in de Rham cohomology
 $H^1(M^3;E_{\operatorname{Ad}\circ\rho_n})$. This form is mapped to the cocycle
\begin{equation}
\label{eqn:integration}
 \begin{array}{rcl}
d_\vartheta:  \pi_1(M^3,p) & \to & \mathfrak{sl}(n,\mathbf C) \\
   {[} \gamma {]}  & \mapsto & \int_{\gamma} \vartheta,
\end{array}
\end{equation}
where $\gamma$ is a loop based at $p$ representing ${[}\gamma{]}\in\pi_1(M^3,p)$.
See \cite[\S 6.3]{Weiss} for details.
The map $d_\vartheta:\pi_1(M^3,p)\to  \mathfrak{sl}(n,\mathbf C) $ is a cocycle, and 
its group cohomology class only depends on the de Rham cohomology class of the form $\vartheta$.

We next  describe  \emph{Weil's construction} \cite{Weil,LubotzkiMagid}, that maps a group cocycle  in $Z^1(\pi_1(M^3);\mathfrak{sl}(n,\mathbf C)_{\operatorname{Ad}\circ\rho_n})$ to an infinitesimal deformation 
of $\rho_n$.
This construction induces an isomorphism between $H^1(\pi_1(M^3);  \mathfrak{sl}(n,\mathbf C)_{\operatorname{Ad}\circ\rho_n} )$ and the Zariski tangent space of 
$X(\pi_1(M^3),\operatorname{SL}(n,\mathbf C))$ at $\chi_n$.
Weil's construction maps the cocycle $d\in Z^1(\pi_1(M^3);\mathfrak{sl}(n,\mathbf C)_{\operatorname{Ad}\circ\rho_n})$ to the first order deformation
 of the representation $\rho_{n}$
\begin{equation}
\label{eqn:weil}
\rho_{n,\varepsilon}(\gamma)= (\operatorname{Id}+\varepsilon\, d(\gamma))\rho_n(\gamma),\qquad \forall \gamma\in\pi_1(M^3,p).
\end{equation}
Since $d(\gamma_1\gamma_2)=d(\gamma_1)+\operatorname{Ad}_{\rho_n(\gamma_1)} d(\gamma_2)$, $\rho_{n,\varepsilon}$ is a first order deformation, namely:
$$
\rho_{n,\varepsilon}(\gamma_1\gamma_2)=\rho_{n,\varepsilon}(\gamma_1)\rho_{n,\varepsilon}(\gamma_2)+ O(\varepsilon^2), \qquad \forall \gamma_1,\gamma_2\in\pi_1(M^3,p).
$$

For an  elementary symmetric polynomial $\sigma_i$, an element $\gamma\in\pi_1(M^3,p)$ and $\vartheta\in \Omega^1(M^3;E_{\operatorname{Ad}\circ\rho_n}) $,
the derivative of $\sigma_i^{\gamma}$ with respect to the direction of the cohomology class of $\vartheta$ is
\begin{equation}
\label{eqn:derivative}
\lim_{\varepsilon\to 0}\frac{ \sigma_i(  (\operatorname{Id}+\varepsilon\, d_{\vartheta}(\gamma))\rho_n(\gamma)   ) -  \sigma_i(\rho_n(\gamma)   ) }{\varepsilon},
\end{equation}
where $d_{\vartheta}$ is as in \eqref{eqn:integration}.

\medskip

\emph{Fix $\gamma\in\pi_1(U)$ a nontrivial peripheral element}. We may assume that the  lift of its holonomy is 
$$
\pm \begin{pmatrix}
     1 & 1 \\
      0 & 1
    \end{pmatrix}.
$$

\begin{rem}
\label{rem:only+}
To simplify, we shall only deal with the case when the lift is $+
\begin{pmatrix}
     1 & 1 \\
      0 & 1
    \end{pmatrix}
$. When it is $-
\begin{pmatrix}
     1 & 1 \\
      0 & 1
    \end{pmatrix}
$, the argument is completely similar. 
\end{rem}

We want to compute the derivatives of $\sigma_i^{\gamma}$ with respect to the forms of Lemma~\ref{Lemma:basis}.

\begin{lemma}
\label{lemma:zeroderivative}
For a peripheral element $\gamma\in\pi_1(U)$, 
the derivative of  $\sigma_i^{\gamma}$ with respect to $dx\otimes   {\mathfrak {h}}_+^j$ is zero.
\end{lemma}

\begin{proof} 
When $\vartheta=dx\otimes   {\mathfrak {h}}_+^j$, $d_{\vartheta}(\gamma)= {\mathfrak{h}}_+^j$, by \eqref{eqn:integration}. Therefore
$\rho_{n,\varepsilon}(\gamma)= (\operatorname{Id}+ \varepsilon  \, {\mathfrak {h}}_+^j)\rho_n(\gamma) $ is upper triangular with $1$ on the diagonal.
In particular $\sigma_i(\rho_{n,\varepsilon}(\gamma))$ is independent of $\varepsilon$, and we get zero when computing the limit \eqref{eqn:derivative}.
\end{proof}

Now, to analyze the derivative of the $\sigma_j^{\gamma}$ with respect to $\omega_i$, the differential forms of Lemma~\ref{Lemma:basis}, we shall look
at characteristic polynomials. Let $P_{i,\varepsilon}^\gamma(\lambda)$ denote the characteristic polynomial of 
$
(\operatorname{Id}+\varepsilon\, d_{\omega_i}(\gamma))\rho_n(\gamma)
$:
$$
P_{i,\varepsilon}^\gamma(\lambda)=\det \big(\lambda \operatorname{Id} - [\operatorname{Id}+\varepsilon\, d_{\omega_i}(\gamma)]\rho_n(\gamma) \big).
$$
We write
$$
P_{i,\varepsilon}^\gamma(\lambda)= (\lambda-1)^n+\varepsilon\, Q_i^\gamma(\lambda) + O(\varepsilon^2),
$$
for some polynomial $Q_i^\gamma(\lambda)\in\mathbf C[\lambda]$. The role of the $Q_i^\gamma(\lambda)$ comes from the following lemma,
whose proof is a consequence of  \eqref{eqn:derivative}.

\begin{lemma}
\label{lemma:coefficientQi} For $i,j=1,\ldots, n-1$,
the $\lambda^{n-j}$-coefficient of $Q_i^\gamma(\lambda)$ is the derivative of $(-1)^j\sigma_j^\gamma$ with respect to
$\omega_i$.  
\end{lemma}

To compute $Q_i^\gamma(\lambda)$ we set the following notation:
\begin{equation*}
A = \lambda \operatorname{Id} - \rho_n(\gamma) \quad\textrm{ and }\quad X_i= d_{\omega_i}(\gamma) \rho_n(\gamma),
\end{equation*}
so that 
$$
P_{i,\varepsilon}^\gamma(\lambda)=\det(A+\varepsilon\, X_i)=\det(A)\det(\operatorname{Id} +\varepsilon A^{-1} X_i).
$$
As the derivative of the determinant at the identity is the trace:
\begin{equation}
\label{eqn:Qi}
Q_i^\gamma(\lambda)=\det(A)\operatorname{trace}(A^{-1} X_i)= (\lambda-1)^n  \operatorname{trace}(A^{-1} X_i).
 \end{equation}
With this formula we may prove:

\begin{proposition} \label{prop:qi_multiple} 
	For $\gamma\in \pi_1(U)$ nontrivial and for $i=1,\ldots,n-1$ the following assertions hold:
	\begin{enumerate}
		\item $Q_i^\gamma(0) = 0$.
		\item $Q_i^\gamma(\lambda)$ is a multiple of $(\lambda-1)^{n-i-1}$ but not of $(\lambda-1)^{n-i}$.
	\end{enumerate}
\end{proposition}
\begin{proof}
	At $\lambda = 0$ we have, $P_{i,\varepsilon}^\gamma(0) = (-1)^n + O(\varepsilon^2)$.
	Indeed the trace of the matrix $d_{\omega_i}(\gamma)$ is zero, and hence
	\[
	\det(\operatorname{Id}+\varepsilon\, d_{\omega_i}(\gamma) ) = 1 + O(\varepsilon^2).
	\]
	This proves the first assertion. In order to prove the second assertion we use  \eqref{eqn:Qi}. 
	To compute $A^{-1}$, as we assume that the holonomy of $\gamma$ is 
	$\left(\begin{smallmatrix} 1 & 1 \\ 0& 1\end{smallmatrix}\right)$, using \eqref{eqn:sym(m)}
	we write
	\begin{equation}
	\label{eqn:N}
	 N= \rho_n(\gamma)-\operatorname{Id}=  \sum_{j=1}^{n-1} \frac{1}{j!} h_+^j, 
	\end{equation}
	so that 
        \[
	   A= (\lambda-1)\operatorname{Id}- N=(\lambda - 1)\left(\operatorname{Id} -(\lambda -1) ^{-1} N\right).
        \]
	As $N^n=0$, the inverse of $A$ is
	\[
	  A^{-1}= (\lambda -1)^{-1}\sum_{k=0}^{n-1}(\lambda-1)^{-k} N^k,
	\]
	and  \eqref{eqn:Qi} becomes:
	\begin{equation}
	 \label{eqn:qin}
	  Q_i^\gamma(\lambda) =  \sum_{k= 0}^{n-1} (\lambda -1)^{n -k-1} \operatorname{trace}(N^k X_i).
	\end{equation}
	On the other hand, by construction of $\mathfrak{h}_-$ \eqref{eqn:h-},
	$$
	\mathfrak{h}_-^{i} =
	\begin{pmatrix}
		0 & 0 & \cdots &  0 &  \cdots & 0\\
		\vdots &  &  &   & & \vdots \\
		a_{i+1,1} & 0 &    &   &  & 0\\
		0 & a_{i+2,2} &    &  &  & 0\\
		\vdots &  & \ddots      &  & &  \vdots \\
		0 & 0 & \cdots  & a_{n,n-i} & \cdots & 0 
	\end{pmatrix},
	$$
	with 
	$
	    a_{i+1,1},a_{i+2,2},\ldots,  a_{n,n-i} >0.
	$
	In addition, since $X^i$ is obtained from $\mathfrak{h}_-^{i} $ by multiplication by upper triangular matrices that have $1$ in the diagonal
	(see  \eqref{eqn:integration} and Lemma~\ref{Lemma:basis}), $X_i$ has the 
	same bottom left $(n-i)$-triangular corner as $\mathfrak{h}_-^{i}$:
	\begin{equation}
	\label{eqn:shapem'}
	  X_i=
	\begin{pmatrix}
		* & * & \cdots & *   & \cdots & *\\
		\vdots &  &  &  & &  \vdots \\
		a_{i+1,1} & * &   &  & &  *\\
		0 & a_{i+2,2} &   & & &   *\\
		\vdots &  & \ddots    & & &  \vdots \\
		0 & 0 & \cdots  & a_{n,n-i} & \cdots & *
	\end{pmatrix},
	\end{equation}
	with $a_{i+1,1},a_{i+2,2},\ldots,  a_{n,n-i} >0$.
	In addition, by \eqref{eqn:N}
\begin{equation}
	\label{eqn:shapeM}
	  N^k=
	\begin{pmatrix}
		 0      & \cdots & b_{1,k+1} &      *    & \cdots  &    *      \\
		 0      &        &    0      & b_{2,k+1} &         &    *      \\
		 \vdots &        &           &           & \ddots  & \vdots    \\
		 0      &        &           &           &         & b_{n-k,n} \\
		\vdots  &        &           &           &         & \vdots    \\
		 0      & \cdots &    0      &      0    & \cdots  &   0
	\end{pmatrix},
	\end{equation}
        with $b_{1,k+1}=a_{n-k,n}>0, \ldots,  b_{n-k,n}=a_{k+1,1} >0$.
        Using the description of $N^k$ and $X_i$:
	\[
	     \operatorname{trace}(N^k X_i)=
	      \left\{
			\begin{array}{ll}
			 0, & \textrm{for } k \geq i +1;  \\
			\operatorname{trace}(h_+^i h_-^i)> 0, & \textrm{for } k =i. 
			\end{array}
	      \right.
	\]
The second assertion follows from this computation and \eqref{eqn:qin}.
\end{proof}

\begin{corollary} \label{coro:basis}
	The polynomials $Q_1^\gamma(\lambda),\ldots ,Q_{n-1}^\gamma(\lambda)$ 
	form a $\mathbf C$-basis for the subspace of polynomials in $\mathbf C[\lambda]$ of degree $\leq n-1$ 
	that are multiples of $\lambda$.
\end{corollary}

\begin{proof} 
	By the first assertion of Proposition \ref{prop:qi_multiple}, it is enough to prove that the polynomials $Q_i^\gamma(\lambda)$ are linearly 
	independent. Assume that for some $\alpha_1,\ldots,\alpha_{n-1}\in\mathbf C$
	\[
	\sum_{i = 1}^{n-1} \alpha_i Q_i^\gamma(\lambda) = 0.
	\]
	The second assertion of Proposition \ref{prop:qi_multiple} implies that
	reduction modulo $\lambda -1$ yields $\alpha_{n-1} = 0$, reduction modulo $(\lambda -1)^2$ yields $\alpha_{n-2} = 0$,
	and so on. Thus the above linear combination must be trivial, as we wanted to prove.
\end{proof}

\begin{proof}[Proof of Theorem~\ref{thm:main}]
 By Corollary~\ref{coro:basis} and Lemma~\ref{lemma:coefficientQi}, 
the $(n-1)\times (n-1)$ matrix whose $(i,j)$-entry is the derivative of $(-1)^j\sigma_j^\gamma$ with respect to $\omega_i$ has nonzero determinant. 
Combining this with Lemma~\ref{lemma:zeroderivative} and Proposition~\ref{prop:generatorimage}, it follows that the 
differential forms $d \, \sigma^{\gamma}_1,\ldots,d \, \sigma^{\gamma}_{n-1}$ form a basis for the cotangent space
of $X(M^3,SL(n,\mathbf C))$ at $\chi_n$. 
Hence Theorem~\ref{thm:main} follows from the holomorphic implicit function theorem.
\end{proof}

\

\begin{footnotesize}
\bibliographystyle{plain}



\textsc{Departament de Matem\`atiques, Universitat Aut\`onoma de Barcelona.}

\textsc{08193 Bellaterra, Spain}

{pmenal@mat.uab.cat, porti@mat.uab.cat}

\end{footnotesize}

 \end{document}